\documentclass{amsart}
\usepackage{amsmath}
\usepackage{amsfonts}
\usepackage{amsthm}
\usepackage{comment}
\usepackage{url}
\usepackage{color}
\usepackage{enumitem}
\setlist[enumerate]{font=\normalfont}
\usepackage{mathtools}
\usepackage[shortalphabetic]{amsrefs}
\usepackage{latexsym}
\usepackage{amssymb}
\usepackage{graphicx}        
\usepackage{float}
\usepackage{youngtab}
\usepackage[colorinlistoftodos]{todonotes}
\usepackage[colorlinks=true, allcolors=blue]{hyperref}
\usepackage{subcaption}

\usepackage{eucal}
\usepackage{xcolor}
\theoremstyle{plain}
\newtheorem*{theorem*}{Theorem}
\newtheorem{theorem}{Theorem}[section]
\newtheorem{corollary}[theorem]{Corollary}
\newtheorem{lemma}[theorem]{Lemma}

\theoremstyle{definition}
\newtheorem{definition}[theorem]{Definition}
\newtheorem{construction}[theorem]{Construction}

\newtheorem*{example*}{Example}

\theoremstyle{remark}
\newtheorem{remark}[theorem]{Remark}
\newtheorem*{remark*}{Remark}

\newcommand{\RR}{\mathbb{R}}

\newcommand{\NN}{\mathbb{N}}

\DeclareMathOperator{\calF}{\mathcal{F}}
\DeclareMathOperator{\calG}{\mathcal{G}}

\title{Harmonic hierarchies for polynomial optimization.}
\date{}

\begin{document}
\author{Sergio Cristancho}
\address{Sergio Cristancho, Departamento de
  Matem\'aticas\\ Universidad de los Andes\\ Carrera 1 No. 18a 10\\ Edificio
  H\\ Primer Piso\\ 111711 Bogot\'a\\ Colombia} 
\email{se.cristancho@uniandes.edu.co}
\author{Mauricio Velasco}
\address{Mauricio Velasco, Departamento de
  Matem\'aticas\\ Universidad de los Andes\\ Carrera 1 No. 18a 10\\ Edificio
  H\\ Primer Piso\\ 111711 Bogot\'a\\ Colombia}
\email{mvelasco@uniandes.edu.co}


\begin{abstract} We introduce novel polyhedral approximation hierarchies for the cone of nonnegative forms on the unit sphere in $\mathbb{R}^n$ and for its (dual) cone of moments. We prove computable quantitative bounds on the speed of convergence of such hierarchies. We also introduce a novel optimization-free algorithm for building converging sequences of lower bounds for polynomial minimization problems on spheres. Finally some computational results are discussed, showcasing our implementation of these hierarchies in the programming language Julia. 
\end{abstract}

\maketitle{}

\section{Introduction}

One of the most basic problems of modern optimization is trying to find the minimum value $\alpha^*$ of a multivariate polynomial $f(x)$ over a compact set $S\subseteq \RR^n$.
Its importance stems from at least two sources: because it serves as a rich model for non-convex global optimization problems and because it has a wealth of applications to which entire books have been devoted~\cites{LBook,LBook2,LBook3,BPT}. A possible approach for solving such problems, pioneered by Shor, Parrilo and Lasserre proposes reformulating them as optimization problems over the cone $P_S$ of polynomials of the same degree as $f$ which are nonnegative on the set $S$, obtaining $\alpha^*$ as
\[\alpha^*=\sup\left\{\lambda\in \RR : f(x)-\lambda\in P_S\right\}.\]
The success of this approach depends on having a description of $P_S$ suitable for optimization. Although exact descriptions of the cone $P_S$ are known for a few sets $S$, (see~\cite{BGP, BSV, BSV2,BSiV}) the most common and practically successful strategy has been the construction of inner (resp. outer) approximation hierarchies for $P_S$ (see for instance~\cite{Pthesis, Lhierarchy, LOther, DeWolff, Venkat, AAA}). An inner (resp. outer) approximation hierarchy is a collection of convex cones $(C_j)_{j\in \mathbb{N}}$ which are contained in $P_S$ (resp. contain $P_S$) and converge to $P_S$ in the sense that the equality $\overline{\bigcup_{j=0}^{\infty} C_j}=P_S$ holds (resp. $\bigcap_{j=0}^{\infty} C_j=P_S$ holds). If the cones $C_j$ form a converging hierarchy then the real numbers
\[\alpha_j:=\sup\left\{\lambda: f(x)-\lambda \in C_j\right\}\]
converge to $\alpha^*$ as $j\rightarrow \infty$ and can be much easier to compute than $\alpha^*$ if the $C_j$ are chosen to be highly structured convex sets such as polyhedra, spectrahedra or their projections. 

The purpose of this article is to introduce several new polyhedral converging hierarchies for approximating the cones $P_{2k}$ of forms of degree $2k$ in the variables $x_1,\dots, x_n$  which are nonnegative on the unit sphere $S\subseteq \RR^n$ and to give quantitative bounds on their rates of convergence.  We call them {\it harmonic hierarchies} because they are closely related with harmonic analysis on spheres (or equivalently with the representation theory of the group $SO(n)$).  

In order to describe our results precisely we need  two preliminary concepts: {\it cubature rules} and {\it polynomial averaging operators} and thus begin by briefly recalling their definitions. Let $R:=\RR[x_1,\dots,x_n]$ be the ring of polynomials with real coefficients, let $R_{k}\subseteq R$ be the subspace of homogeneous polynomials of degree $k$ and let $\mu$ be the $(n-1)$-dimensional area measure on the sphere $S\subseteq \RR^n$. Recall that a {\it cubature rule} of algebraic degree $2t$ for $\mu$ is a pair $(X,W)$ where $X\subseteq S$ is a finite set  and $W: X\rightarrow \RR_{>0}$ is a nonnegative function for which the following equality holds 
\[\forall f\in R_{2t}\left(\int_S f(y) d\mu(y) = \sum_{x\in X} W(x)f(x)\right).\]
If $g(t)$ is a univariate polynomial which is nonnegative on the interval $[-1,1]$, we define its {\it polynomial averaging operator} $\Gamma_g: R\rightarrow R$ by the convolution formula
\[\Gamma_g(f)(x):=\int_{S}g\left(\langle x,y\rangle \right)f(y)d\mu(y)\]

Our first result shows that the interplay of cubature rules and averaging operators can be used to construct polyhedra inside $P_{2k}$,

\begin{theorem}\label{Thm: polyhedra} Let $h(t)=a_0+a_2t^2+\dots+ a_{2s}t^{2s}$ be an even univariate polynomial which is nonnegative on $[-1,1]$ and let $k$ be a positive integer. Define the linear map $\hat{\Gamma}_h: R_{2k}\rightarrow R_{2k}$ by the formula
\[\hat{\Gamma}_h(f)=\sum_{j=0}^s a_{2j} \frac{\int_S\langle x,y\rangle^{2j} f(y)d\mu(y)}{\|x\|^{2(j-k)}}.\]
If $Q$ is the set of polynomials in $R_{2k}$ that have nonnegative values at all points $X$ of a cubature rule $(X,W)$ of algebraic degree $2(s+k)$, then the set $A:=\hat{\Gamma}_h(Q)$ is a polyhedral cone in $R_{2k}$ and the inclusion $A\subseteq P_{2k}\subseteq Q$ holds.
\end{theorem}

The previous theorem is a convenient method to produce polyhedra inside $P_{2k}$ because, as observed by Blekherman~\cite{BConv}, the averaging maps $\hat{\Gamma}_h$ can be diagonalized explicitly, allowing their efficient computation. This property occurs because the maps $\Gamma_h$ are $SO(n)$-equivariant and thus become diagonal in the harmonic basis. More precisely, recall that every homogeneous polynomial $f\in R_{2k}$ can be written uniquely in its {\it harmonic expansion} as
\[f= \|x\|^{2k}f_0 + \|x\|^{2(k-1)}f_2 + \|x\|^{2(k-2)}f_4 +\dots+ f_{2k}\]
where the $f_{2j}$ are homogeneous harmonic polynomials of degree $2j$ (see Section~\ref{Sec: Harmonic} for details). Using this decomposition, the operators $\hat{\Gamma}_h$ take the following particularly simple form,

\begin{lemma}\label{lem: diagonal} Let $h(t)=\sum_{j=0}^n \lambda_{2j} g_{2j}(t)$ be the unique expression of $h(t)$ as linear combination of Gegenbauer polynomials (suitably normalized as in Definition~\ref{def: normGegenbauer}). If 
\[f= \|x\|^{2k}f_0 + \|x\|^{2(k-1)}f_2 + \|x\|^{2(k-2)}f_4 +\dots+ f_{2k}\]
is the unique harmonic expansion for $f\in R_{2k}$ then the equality
\[\hat{\Gamma}_h(f)=\lambda_0\|x\|^{2k}f_0 + \lambda_2\|x\|^{2(k-1)}f_2 + \lambda_4\|x\|^{2(k-2)}f_4 +\dots+ \lambda_{2k}f_{2k}\]
holds.
\end{lemma}
We can now introduce the main construction of this article
\begin{construction}[Linear Harmonic Hierarchies]\label{Const: HH}
Given:
\begin{enumerate}
\item Cubature rules $(X_{2t},W_{2t})$ for $\mu$ of algebraic degree $2t$ for every integer $t$ and
\item A sequence of univariate polynomials $(h_s(t))_{s\in \NN}$ which are nonnegative on the interval $[-1,1]$. 
\end{enumerate}
define the {\it linear harmonic hierarchy} determined by $(1)$ and $(2)$ in degree $2k$ as the sequence of polyhedra $(A_s)_{s\in \NN}$ given by $A_s:=\hat{\Gamma}_{h_s}(Q_s)$ where $d_s:=\deg(h_s)$,
\[Q_{s}:=\left\{F\in R_{2k}: \forall x\in X_{2(k+d_s)}\left(F(x)\geq 0\right)\right\},\]
and $\hat{\Gamma}_{h_s}: R_{2k}\rightarrow R_{2k}$ denotes the averaging operator determined by the polynomial $h_s$, defined in Theorem~\ref{Thm: polyhedra}.
\end{construction}

Our main result gives quantitative convergence bounds for harmonic hierarchies. Such bounds are expressed in terms of the {\it Frobenius threshold of a polynomial $h(t)$ in degree $2k$}, defined as the Frobenius norm of the operator $\hat{\Gamma}_h^{-1}-I: R_k\rightarrow R_k$ or, using the notation of Lemma~\ref{lem: diagonal}, as the quantity
\[\tau_{2k}(h):=\sqrt{\sum_{j=0}^{2k} \dim(H_{2j})\left(\frac{1}{\lambda_{2j}}-1\right)^2}.\]
where ${\rm dim}(H_j)$ denotes the dimension of the space of harmonic polynomials of degree $j$ in $\RR^n$.

\begin{theorem}\label{Thm: bounds} The Harmonic Hierarchies introduced in Construction~\ref{Const: HH} have the following properties:
\begin{enumerate}
\item The sets $(A_s)_{s\in \mathbb{N}}$ are polyhedral cones satisfying $A_s\subseteq P_{2k}\subseteq R_{2k}$ for every integer $s$.
\item Assume $\hat{\Gamma}_{h_s}: R_{2k}\rightarrow R_{2k}$ is invertible. If $f\in R_{2k}$ satisfies the inequality
\[\min_{x\in X_{2(k+d_s)}} f(x) > \frac{\tau_{2k}(h_s)}{\sqrt{\mu(S)}}\|f\|_2\]
then $f\in A_s$.
\item If $\lim_{s\rightarrow \infty}\tau_{2k}(h_s)=0$ then every strictly positive polynomial in $R_{2k}$ is contained in some $A_s$ and in particular the hierachy is convergent in the sense that the following equality holds
\[P_{2k} = \overline{\bigcup_{s=0}^{\infty}A_s}.\] 
\end{enumerate}
\end{theorem}

In Corollary~\ref{cor: new_quad} below we give an explicit cubature formula of algebraic degree $2t$ on $S\subseteq \RR^n$ supported on $2(t+1)^{n-1}$ points for every positive integer $t$ which allows us to build harmonic hierarchies for any sequence of polynomials $(h_s)_{s\in \mathbb{N}}$. The following Corollary describes the quantitative behavior of such hierarchies for two different choices of the sequence $(h_s)_s$. The delicate convergence estimates involved are contained in work of Blekherman~\cite{BConv} and Fang-Fawzi~\cite{FangFawzi} further discussed in Section~\ref{Sec: KernelSelection}.

\begin{corollary}\label{cor: various_g} The following statements hold:
\begin{enumerate}
\item If $h_s(t):=\frac{t^{2s}}{\int_Sy^{2s}d\mu(y)}$, then for every integer $k$ the following inequality holds:
\[
\frac{1+\frac{n}{2}}{s}+O\left( {\textstyle\frac{1}{s^2} }\right) \leq \tau_{2k}(h_s) \leq  D_{2k}\frac{k^2+\frac{kn}{2} }{s} + O\left( \tfrac{1}{s^2} \right),
\]
where $D_{2k}=\max_{j=0,\dots, k} \dim(H_{2j})$.
\item If  $
h_s(t)=q_s(t)^2=\sum_{j=0}^{2s} \lambda_j g_j(t)$,
where $q_s(t) = \sum_{j=0}^s \eta_j g_j(t)$ is the solution to 
\[
\rho^*_{2k,s} = \min_{q_s,\lambda_0=1} \sum\limits_{j=0}^k (1-\lambda_{2j}),
\]
then for every integer $k$ the following inequality holds:
\[
\tau_{2k}(h_s) \leq \sqrt{D_{2k}}k^{2} n^2 O\left(\tfrac{1}{s^2}\right).
\]
\end{enumerate}
In particular the harmonic hierarchies $(A_s)_{s\in \mathbb{N}}$ determined by both sequences $(h_s)_{s\in \mathbb{N}}$ converge to $P_{2k}$ as $s\rightarrow \infty$ in either case.
\end{corollary}
As the previous result shows, the choice of the polynomials $(h_s)_{s\in \mathbb{N}}$ has a significant effect on the quality of approximation of $A_s\subseteq P_{2k}$. In Section~\ref{Sec: KernelSelection} we contribute to this central issue by proving (see Theorem~\ref{Thm: Optimal_g}) that the problem of finding an optimal kernel $h$ (in the sense that $\tau_{2k}(h)$ is minimal, among all valid $h$ of degree $2s$) is a convex optimization problem over a spectrahedron and thus amenable to computation.

Furthermore in Section~\ref{Sec: optimization-free} we introduce a novel optimization-free algorithm for polynomial minimization on the sphere which arises naturally from minimizing polynomials via Harmonic Hierarchies.

In Section~\ref{Sec: HHMoments} we adopt a dual point of view and define harmonic hierarchies for moments. More precisely, by Tchakaloff's Theorem the cone $P_{2k}^*\subseteq R_{2k}^*$ dual to $P_{2k}$ captures the moments of degree $2k$ of all Borel measures on the sphere $S$ in the sense that $P_{2k}^*$ consists precisely of those linear operators $\ell: R_{2k}\rightarrow \RR$ which satisfy
\[\forall f\in R_{2k}\left(\ell(f)=\int_S f(y)d\nu(y)\right)\]
for some Borel measure $\nu$ on $S$. Our final Theorem provides {\it harmonic hierarchies for moments}, that is a sequence of polyhedra $(A_s^*)_{s\in \mathbb{N}}\subseteq R_{2k}^*$ giving a converging hierarchy of outer approximations for the cone $P_{2k}^*$ of moments. 

\begin{construction}[Outer Harmonic Hierarchies for Moments]\label{const: HH_dual}
Given:
\begin{enumerate}
\item Cubature rules $(X_{2t},W_{2t})$ for $\mu$ of algebraic degree $2t$ for every integer $t$ and
\item A sequence of univariate polynomials $(h_s(t))_{s\in \NN}$ which are nonnegative on the interval $[-1,1]$,
\end{enumerate}
define the {\it harmonic hierarchy for moments} determined by $(1)$ and $(2)$ in degree $2k$ as the sequence of polyhedra $(A_s^*)_{s\in \NN}$ where $A_s^*\subseteq R_{2k}^*$ is defined as the convex hull of the set of operators
\[L_{y}:=\left\langle \sum_{j=0}^k \lambda^{(s)}_{2j}\|x\|^{2(k-j)}\phi^{2j}_y(x),\bullet\right\rangle\]
for $y\in X_{2(k+d_s)}$ where the $\lambda_{2j}^{(s)}$ are the coefficients of $h_s$ in its Gegenbauer expansion (as in Lemma~\ref{lem: diagonal}) and $\phi_y^{2j}(x)$ is the homogeneous polynomial which represents the evaluation at $y$ (see Theorem~\ref{Thm: Zonal} for explicit formulas for $\phi_y^{2j}(x)$ in terms of Gegenbauer polynomials).
\end{construction}

Our next result summarizes the basic properties of harmonic hierarchies for moments.

\begin{theorem}\label{Harmonic_Moments}

The following statements hold:
\begin{enumerate}

\item The sets $(A_s^*)_{s\in \mathbb{N}}$ are polyhedral cones satisfying $R_{2k}^*\supseteq A_s^*\supseteq P_{2k}^*$. Furthermore $A_s^*$ is the dual cone to $A_s$.
\item If $\lim_{s\rightarrow \infty} \tau_{2k}(h_s)=0$ then the hiererachy $(A_s^*)_{s\in \mathbb{N}}$ converges to $P_{2k}^*$ in the sense that the following equality holds
\[\bigcap_{s=0}^{\infty} A_s^* = P_{2k}^*.\]
\end{enumerate}
\end{theorem}

Finally in Section~\ref{Sec:Julia} we introduce our Julia package for Harmonic Hierarchies (available at~\url{github}) and show some simple computational results obtained with it. We showcase our ``optimization-free" algorithm for polynomial minimization on the sphere via harmonic hierarchies and verify that its practical behavior is similar to what our theory predicts.  Applications of Theorem~\ref{Harmonic_Moments} and the extension of our package for solving problems expressible via the method of moments will be the object of upcoming subsequent work.

\subsection{Relationship with previous work.} 
The notion that cubature rules should play a useful role in polynomial optimization appears in~\cite{Piazzon-Vianello-Q, Piazzon-Vianello-CGrid} where the authors propose constructing upper bounds for the minimum value $\alpha^*$ of a polynomial by evaluating it at the nodes of a cubature rule. It is shown in~\cite{Piazzon-Vianello-Q} that this ``optimization-free" approach is at least as good as the SDP approach proposed in~\cite{LOther} for polynomial optimization (see Remark~\ref{Rem:cubature} for details). In the language of this article, their work proposes an outer hierarchy of approximation for $P_{2k}$ via the polyhedra $Q_s$ defined in Construction~\ref{Const: HH}. By contrast, our work provides {\it inner} approximations for $P_{2k}$ providing {\it lower bounds} on the minima of polynomials as well as a novel optimization-free approach (see Section~\ref{Sec: optimization-free}). Lower bounds on $\alpha^*$ are typically harder to obtain and more valuable since they involve proving a statement with a universal quantifier.

The results of Fang and Fawzi in~\cite{FangFawzi} are the best estimates that are currently available on the speed of convergence of the sum-of-squares hierachy for polynomial optimization on the sphere. In this article we show that the exact same bounds apply to our {\it linear approximation hierarchies} and provide novel quantitative convergence bounds which depend on more readily computable quantities. It would be interesting to extend harmonic hierarchies to other spaces such as the hypercube, the ball and the simplex for which we have natural measures and explicit formulas for the reproducing kernel leveraging the ideas of Slot-Laurent~\cite{LaurentSlot} and Slot~\cite{Slot}. 

In~\cite{Alperen} Erg\"ur constructs {\it random} polyhedral approximation hierarchies for the cone of nonnegative polynomials. More precisely, the author builds a family of random polytopes which approximates the cone of nonnegative polynomials lying in a given subspace $E$ within a specified scaling constant {\it with high probability} (see ~\cite[Corollary 6.5]{Alperen} for precise statements). Remarkably, the author shows that the number of facets in such approximations depends explicitly on the dimension of the subspace and can be much better for sparse nonnegative polynomials than for abitrary nonnegative polynomials. While our approximation hierarchies are deterministic and explicit they do not take into account the sparsity structure of our target polynomials. Developing an extension of harmonic hierarchies which can incorporate sparsity is an interesting open problem.

{\bf Acknowledgments.} We wish to thank Greg Blekherman for many stimulating conversations which motivated us to pursue this work. We thank Alex Towsend for pointing us to recent ideas on Gaussian quadrature computation and their high quality implementations. We thank Monique Laurent, Lucas Slot and Alperen Erg\"ur for various references and useful feedback on earlier versions of the results contained in this article.

\section{Cubature formulas}\label{Sec:cubature}

By a {\it cubature formula of algebraic degree $2t$} for $\mu$ on $S\subseteq \RR^n$ we mean a pair $(X,W)$ where $X\subseteq S$ is a finite set and $w:X\rightarrow \RR_{>0}$ is a function with strictly positive values which satisfy the equality
\[ \int_S f(y)d\mu(y) = \sum_{x\in X} W(x) f(x) \]
for every homogeneous polynomial (i.e., form) $f\in R_{2t}$.

The main invariant of a cubature formula is its {\it size} $|X|$. From Caratheodory's Theorem we know that there exist cubature rules of strength $2t$ of size at most $\binom{2t+n-1}{t}+1$ and it is easy to see that no cubature formula of strength $2t$ and size less than $\binom{t+n-1}{t}$ can exist, since otherwise the square of a form vanishing at all points of $X$ would fail to satisfy the equality above (this lower bound is known to be strict on the sphere if $n,t>2$~\cite{Taylor}). Despite a very significant amount of work (see for instance the surveys~\cite{Stroud,Cools1,Cools2}) and the fact that such formulas could have a wealth of applications no general formula is known for producing cubature rules of given weight and (provably) minimal size on the sphere (see~\cite[pg. 294-303]{Stroud} for formulas in some special cases).

\subsection{An explicit cubature rule for spheres}\label{Sec: Gauss-product}

In this section we give explicit cubature rules of arbitrary even algebraic degrees on the sphere $S\subseteq \RR^n$. We will use well-known formulas of Gauss-product type~\cite[pg.40-43]{Stroud} for which we include a self-contained treatment for the reader's benefit. Such product formulas can be combined with recent ideas on fast Gauss-Jacobi quadrature computation~\cite{Hale-Townsend} to produce highly accurate cubature rules very efficiently. Such rules are key components in our implementation of harmonic hierarchies (see Section~\ref{Sec:Julia}).

We denote the points of $\RR^{n}$ by pairs $(s,\zeta)\in \RR\times \RR^{n-1}$. Recall~\cite[Theorem A.4, pg.242]{Axler} that if $f$ is an integrable, Borel-measurable function on the sphere $S^{n-1}\subseteq \RR^{n}$ then the following equality holds:
\begin{equation}
\label{eqn: int}
\int_{S^{n-1}}fd\mu = \int_{-1}^1\left(1-s^2\right)^{\frac{n-3}{2}} \left(\int_{S^{n-2}} f\left(s, \sqrt{1-s^2}\zeta\right)d\mu(\zeta)\right) ds
\end{equation}

We will use the product structure of formula~(\ref{eqn: int}) to inductively construct explicit cubature rules on spheres of every dimension and even strength which are invariant under sign changes. Recall that the group of sign changes in $\RR^n$ consists of linear transformations $T:\RR^n\rightarrow \RR^n$ which send $(x_1,\dots, x_n)$ to $(\epsilon_1x_1,\dots, \epsilon_nx_n)$ with $\epsilon_i\in \{-1,1\}$ and $i=1,\dots, n$. A cubature rule $(X,W_X)$ on $S^{n-1}$ is invariant under sign changes if for every $x\in X$ and every sign change $g$ we have $gx\in X$ and $W(gx)=W(x)$. An important ingredient of the construction will be the Gauss-Jacobi quadrature rules on the interval $[-1,1]$ for a given weight function $w(y)=(1+y)^{\alpha}(1-y)^\beta$ so we begin by recalling their definition. If $\alpha,\beta>-1$ are given and $X:=\{x_1,\dots, x_t\}\subseteq [-1,1]$ is the set of roots of the Jacobi polynomial $P_t^{(\alpha,\beta)}(x)$ then there exists an explicit function $W: X\rightarrow \RR_+$ (see~\cite[pg. 352]{Szego} or~\cite[1.4]{Hale-Townsend} for an explicit formula) such that the equality
\[\int_{-1}^1 f(y)w(y) dy = \sum_{x\in X} W(x)f(x)\]
hols for every univariate polynomial $f(t)$ of degree $2t-1$ or less. 

\begin{construction} \label{Const: Spherequad} Suppose that $(Y,W_Y)$ is a cubature on $S^{n-2}$ and that $(Z,W_Z)$ is a Gaussian quadrature rule for the weight function $w(s)=(1-s^2)^{\frac{n-3}{2}}$ on $[-1,1]$. Define the pair $(X,W_X)$ on $S^{n-1}$ via the formulas: 
\[X=\left\{\left(z, \sqrt{1-z^2}y\right): (z,y)\in Z\times Y\right\}\]
\[W_X\left(z, \sqrt{1-z^2}y\right) :=W_Z(z)W_Y(y)\]
\end{construction}
The following Theorem summarizes the main properties of this construction
\begin{theorem} If $(Y,W_Y)$ and $(Z,W_Z)$ have algebraic degree $2t$ and $(Y,W_Y)$ is invariant under sign changes then the pair $(X,W_X)$ is a cubature rule of algebraic degree $2t$ in $S^n$ which is invariant under sign changes. Furthermore $|X|=|Z||Y|$.
\end{theorem}
\begin{proof} Since the Jacobi polynomials satisfy the symmetry relation $P_t^{(\alpha,\beta)}(-z)=(-1)^tP_t^{(\beta,\alpha)}(-z)$ and we are in the $\alpha=\beta$ case we conclude that the nodes of the Gaussian cubature $(Z,W_Z)$ are closed under multiplication by $(-1)$. Furthermore the equality $W_Z(-x_j)=W_Z(x_j)$ holds because the explicit formula for the Gaussian cubature weights from~\cite[pg. 352]{Szego} depends on the value of the derivative only through its square. We conclude that $(X,W_X)$ is invariant under sign change of the first component. Furthermore if $g$ is the transformation changing the sign of any component with index at least two then $g(z,\sqrt{1-z^2}y)=(z,\sqrt{1-z^2}g(y))$. Since $Y$ is invariant under sign changes we conclude that $(z,\sqrt{1-z^2}g(y))$ lies in $X$ and furthermore we know $W_Y(y)=W_Y(g(y))$ which implies that $W_X(z,\sqrt{1-z^2}g(y)) = W_X(z,\sqrt{1-z^2}y)$ as claimed.  Now suppose $f(s,\zeta)=s^{a_1}\zeta_1^{b_1}\dots \zeta_{n-1}^{b_{n-1}}$ is a monomial of degree $2t$. If $a_1$ or some $b_i$ is odd then the integral and the cubature rule $(X,W_X)$ have both value zero because the integrand gets multiplied by minus one by the sign change of the coordinate which appears with odd exponent. Thus it suffices to prove the claim for monomials all of whose exponents are even. More precisely suppose $f(s,\zeta)=s^{2a_1}\zeta_1^{2b_1}\dots \zeta_{n-1}^{2b_{n-1}}$ with $2a_1+ 2b_1+\dots + 2b_{n-1}= 2t$.
Now $f(s,\zeta\sqrt{1-s^2})= s^{2a_1}(1-s^2)^{b_1+\dots+b_n} \zeta_1^{2b_1}\dots \zeta_{n-1}^{2b_{n-1}}$. Since as functions on $S^{n-1}$
\[\zeta_1^{2b_1}\dots \zeta_{n-1}^{2b_{n-1}}=\|(\zeta_1,\dots,\zeta_{n-1})\|_2^{2a_1}\zeta_1^{2b_1}\dots \zeta_{n-1}^{2b_{n-1}}\]
and the right-hand side has degree $2t$ we can use the cubature rule $(Y,W_Y)$ to conclude that for every $s\in [-1,1]$
\[\int_{S^{n-2}} f\left(s, \sqrt{1-s^2}\zeta\right)d\mu(\zeta) =  s^{2a_1}(1-s^2)^{b_1+\dots+b_n} \sum_{y\in Y} W_Y(y) y_1^{2b_1}\dots y_n^{2b_n}\]
By integrating with respect to $s$ and using the fact that $(Z,W_Z)$ is a Gaussian cubature rule for polynomials of degree $t$ or less with respect to the weight function $\left(1-s^2\right)^{\frac{n-3}{2}}$ we conclude that

\[\int_{-1}^1\left(1-s^2\right)^{\frac{n-3}{2}} \left(\int_{S^{n-1}} f\left(s, \sqrt{1-s^2}\zeta\right)d\mu_{n-1}(\zeta)\right) ds =\]
\[= \int_{-1}^1\left(1-s^2\right)^{\frac{n-3}{2}} s^{2a_1}(1-s^2)^{b_1+\dots+b_n}\sum_{y\in Y} W_Y(y) y_1^{2b_1}\dots y_n^{2b_n}ds = \]
\[=\sum_{z\in Z}W_Z(z) z^{2a_1}(1-z^2)^{b_1+\dots+b_n} \sum_{y\in Y} W_Y(y) y_1^{2b_1}\dots y_n^{2b_n} =\]
\[=\sum_{y\in Y} \sum_{z\in Z} W_Z(z)W_Y(y) y_1^{2b_1}\dots y_n^{2b_n}z^{2a_1}(1-z^2)^{b_1+\dots+b_n}=\]
\[=\sum_{y\in Y} \sum_{z\in Z} W_Z(z)W_Y(y) f\left(z, \sqrt{1-z^2}y\right)\]
Using Equation~(\ref{eqn: int}) we conclude that for every polynomial of degree $2t$ the equality
\[\int_{S^{n-1}} f d\mu = \sum_{x\in X} W_X(x) f(x)\]
holds as claimed.
\end{proof}

Using the construction iteratively, starting from the cubature rule on the circle $S\subseteq \RR^2$ given by the vertices of a polygon with $2(t+1)$ sides and equal weights we prove

\begin{corollary} \label{cor: new_quad} Construction~\ref{Const: Spherequad} defines a cubature rule of algebraic degree $2t$ consisting of $2(t+1)^{n-1}$ points on the sphere $S\subseteq \RR^n$.
\end{corollary}

\begin{remark} \label{Rem:cubature} Having explicit cubature rules gives a useful procedure for estimating minima of polynomials. As shown in the work of Piazzon et al.~\cite{Piazzon-Vianello-Q}, by letting $\alpha_j^{\rm quad}:=\min_{x\in X_j}f(x)$ be the minimum over the nodes of increasing cubature rules of algebraic degree $j$ we obtain a sequence which approaches $\alpha^*$. 
To see this, recall from~\cite{LOther} that the sequence of minima of the semidefinite programs 
\[\beta_t:=\inf\left\{ \int f(x)g(x)d\mu: \int g(x)d\mu=1\text{ and $g(x)$ is SOS of polys. of degree $t$}\right\}\]
converges to $\alpha^*$ and note that if $k={\rm deg}(f)+t$ then
\[\int f(x)g(x)d\mu = \sum_{z\in X_k}f(z)g(z)\geq \alpha^{\rm quad}_{k}\sum_{z\in X_k} g(z) =\alpha^{\rm quad}_{k}\int g(z)d\mu(z) = \alpha^{\rm quad}_{k}\]
so $\beta_t\geq \alpha^{\rm quad}_{k}\geq \alpha^*$, the $\alpha_k^{\rm quad}$ converge to the optimum at least as fast as the $\beta_t$ and in particular $\alpha^{\rm quad}_{k}-\alpha^*=O(1/k^2)$ by results of De Klerk, Laurent and Zhao~\cite{dKLZ}.
\end{remark}

\section{Harmonic analysis on spheres}
\label{Sec: Harmonic}
\subsection{Reproducing Kernels for spaces of functions on the sphere}
\label{Sec:RKHS}
Suppose that $\calF$ is a finite-dimensional vector space of continuous real-valued functions on the sphere $S\subseteq \RR^n$ and let $\mu$ be the $(n-1)$-dimensional volume measure. The inner product 
\[\langle f,g\rangle:=\int_{S} f(y)g(y) d\mu(y)\] 
makes $\calF$ into a Hilbert space. Every point $x\in S$ defines a linear evaluation map $ev_x: \calF\rightarrow \RR$ which sends a function $f$ to its value $f(x)$ at $x$. Since $\calF$ is a Hibert space the evaluation map is represented by a unique element $\phi_x\in \calF$, meaning that $\forall f \in \calF\left(f(x)= \langle f,\phi_x\rangle\right)$.  The {\it Christoffel-Darboux kernel} (or reproducing kernel) of the Hilbert space $\calF$ is the function $K_{\calF}: S\times S\rightarrow \RR$ given by \[K_{\calF}(x,y)=\langle \phi_x,\phi_y\rangle = \phi_x(y)=\phi_y(x).\]
The following basic Lemma summarizes its main properties:
\begin{lemma}\label{lem: easy_RKHS} The following statements hold for every $x,y\in S$:
\begin{enumerate}
\item The function $K_{\calF}(x,y)$ is symmetric (i.e. $K_{\calF}(x,y)=K_{\calF}(y,x)$) and for every finite collection $x_1,\dots, x_M$ of points of $S$ the matrix $K_{\calF}(x_i,x_j)$ is positive semidefinite. 
\item $K_{\calF}(x,y)$ has the following {\it reproducing property}
\[\forall f\in \calF \forall x\in S \left(f(x)= \int_S K_{\calF}(x,y)f(y)d\mu(y)\right)\]
and furthermore this property specifies $K_{\calF}(x,y)$ uniquely.
\item If $(e_i(x))_i$ is any orthonormal basis for $\calF$ then $K_{\calF}(x,y) = \sum_j e_j(x)e_j(y)$. In particular the equality $\int_{S} K(x,x)d\mu(x)=\dim(\calF)$ holds.
\end{enumerate}
\end{lemma}

In this Section we will describe some distinguished subspaces of functions on the sphere and give explicit formulas for their reproducing kernels.

\subsection{Harmonic decomposition on spheres}
The orthogonal group $G:=SO(n)$ acts on $\RR^n$ by left multiplication and on the ambient polynomial ring $R$ via the resulting contragradient action defined by $\rho^*(g)(f)(x):=f(g^{-1}(x))$. This action respects multiplication and preserves the graded components $R_j$ of $R$. The decomposition of each graded component into $SO(n)$-irreducible subrepresentations is well understood (see~\cite[Theorem 3.1]{Helgason}). For each integer $k$ we have
\[R_{2k} =\bigoplus_{j=0}^k \left(\|x\|^{2(k-j)} H_{2j}\right)\]
where $H_{2j}\subseteq R_{2j}$ is the subspace consisting of homogeneous harmonic polynomials of degree $2j$ (i.e. forms $F$ of degree $2j$ satisfying $\Delta F=0$  where $\Delta = \sum_{i=1}^n\frac{\partial^2}{\partial x_i^2}$ is the laplacian operator). The $H_{2j}$ are pairwise non-isomorphic irreducible representations of $SO(n)$ and as a result, a homogeneous polynomial $f\in R_{2k}$ has a unique {\it harmonic decomposition} 
\[ f= \|x\|^{2k}f_0 + \|x\|^{2(k-1)}f_2 + \|x\|^{2(k-2)}f_4 +\dots+ f_{2k}\] 
with $f_{2j}\in H_{2j}$ for $j=0,1,\dots,k$ (see~\cite[Theorem 5.7]{AAA} for an elementary proof of the existence of this decomposition).
In particular the following equalities hold \[{\rm dim}(H_{2j})={\rm dim}(R_{2j})-{\rm dim}(R_{2(j-1)})=\binom{n+2j-1}{2j}-\binom{n+2j-3}{2j-2}.\] 

\subsection{Reproducing kernels for spaces of harmonic polynomials}
 
If $H_j$ is the subspace of homogeneous polynomials of degree $j$ restricted to $S$ and $y$ is any point of $S$ then the evaluation map $ev_y: H_j\rightarrow \RR$ is fixed by the subgroup $G_y\leq SO(n)$ consisting of those rotations which fix $y$. As a result the harmonic polynomial $\phi_y^{(j)}$ which represents this evaluation on $H_j$ (i.e. which satisfies $\langle f,\phi_y^{(j)}\rangle = f(y)$ for every $f\in H_j$) is fixed under the action of $G_y$ and satisfies the normalizing property appearing in Lemma~\ref{lem: easy_RKHS} part $(3)$. These properties determine the polynomial $\phi_y^j$ uniquely and allow us to obtain an explicit formula in terms of {\it Gegenbauer polynomials}, whose definition we now recall. If $S\subseteq \RR ^n$ and $n\geq 3$ we let $\alpha:=\frac{n-2}{2}$ and define the $j$-th Gegenbauer polynomial $C_j^{(\alpha)}(t)$ recursively by the formulas
\[C_0^{(\alpha)}(t)=1\text{ , }C_1^{(\alpha)}(t)=2\alpha t\text{ and }\]
\[C_j^{(\alpha)}(t)= \frac{1}{j}\left[2t(j+\alpha -1)C_{j-1}^{(\alpha)}(t)-(j+2\alpha-2)C^{(\alpha)}_{j-2}(t)\right]\text{ if $j\geq 3$}.\]

The following Theorem gives formulas for the reroducing kernels on the spaces $H_j$. We provide a sketch of a proof because the argument is simple and beautiful (see~\cite[Theorem 2.24]{Morimoto} for details) and provides a natural motivation for the definition of Gegenbauer polynomials. 

\begin{theorem} \label{Thm: Zonal} For each $y\in S$ and nonnegative integer $j$ there exists a unique polynomial $\phi_y^{j}(x)\in R$ satisfying the following conditions:
\begin{enumerate}
\item $\phi_y^j$ is homogeneous of degree $j$ and harmonic.
\item $\phi_y^j$ is fixed by the action of the stabilizer subgroup $G_y\subseteq O(n)$.
\item $\phi_y^j(y)=\frac{{\rm dim}(H_j)}{\mu(S)}$ 
\end{enumerate}
Furthermore $\phi_y^j$ represents the evaluation at $y$ on $H_j$ and is given, in terms of Gengenbauer polynomials, by the formula
\[\phi_y^j(x)=\frac{{\rm dim}(H_j)}{\mu(S)C_j^{(\alpha)}(1)}\|x\|^{j}C_j^{(\alpha)}\left(\left\langle\frac{x}{\|x\|},y\right\rangle\right)\]
\end{theorem}
\begin{proof} We will show that there is exactly one polynomial satisfying properties $(1)$, $(2)$ and $(3)$. Any $p\in R_j$ can be written as
\[p=\sum_{k=0}^j x_n^k p_{j-k}(x_1,\dots,x_{n-1})\] 
where the $p_{j-k}$ are homogeneous polynomials of degree $j-k$ in the first $(n-1)$ variables. Without loss of generality assume $y=(0,\dots,0,1)$. Since $p$ is fixed by $G_y$ the polynomials $p_{j-k}$ are invariant under arbitrary rotations in $SO(n-1)$ and thus must be scalar multiples of even powers of the norm $(x_1^2+\dots+x_{n-1}^2)$ and in particular $j-k$ is even if $p_{j-k}\neq 0$. Thus we can write
\[p=\sum_{k=0}^{\lfloor\frac{j}{2}\rfloor} x_n^{j-2k}c_k\left(x_1^2+\dots+x_{n-1}^2\right)^{k}\]
for some scalars $c_k$. The equation $\Delta p=0$ then yields the recursive relations
\[2(k+1)(n+2k)c_{k+1} = -(j-2k)(j-2k-1)c_k, \text{ for $k=0,1\dots, j/2-1$}.\]
The constant $c_0$ is uniquely determined by the normalization property $(3)$ above and we have shown existence and uniqueness of the polynomial $p$. Since the polynomial $\phi_y^{(j)}$ which represents evaluation at $y$ on $H_j$ satisfies properties $(1)$ $(2)$ and $(3)$ it must coincide with $p$.  The explicit formula (and the definition of Gegenbauer polynomial) are equivalent to the recursive relations above. 
\end{proof}
Motivated by the previous Theorem we define:
\begin{definition}\label{def: normGegenbauer} The {\it normalized Gegenbauer polynomial} of degree $j$ on the sphere $S\subseteq \RR^n$ is the univariate polynomial given by
\[g_{j}(t) = \frac{{\rm dim}(H_j)}{\mu(S)C_j^{(\alpha)}(1)}C_{j}^{(\alpha)}(t)\]
where $\alpha=\frac{n-2}{2}$ and $C_j^{(\alpha)}(t)$ is the Gegenbauer polynomial defined at the beginning of this Section.
\end{definition}

\subsubsection{An application of reproducing kernels}
As an application of the reproducing kernels for $H_j$ we obtain a well-known sharp bound relating the $L^{\infty}$ and the $L^2$ norm of an arbitrary harmonic polynomial which will be used for obtaining easily computable bounds for Harmonic Hierarchies.

\begin{lemma}\label{lem: sup_bound} If $f\in H_j$ then the following inequality holds
\[\|f\|_{\infty} \leq \sqrt{\frac{{\rm dim}(H_j)}{\mu(S)}}\|f\|_2\]
Furthermore the equality holds if $f(x)=\phi_y^{(j)}(x)$. 
\end{lemma}
\begin{proof} The reproducing property of $\phi_y^{(j)}$ implies that the equality
\[f(y) = \int_S f(x)\phi_y(x)d\mu(x)\] 
holds for $f\in H_j$. By the Cauchy-Schwartz inequality this implies
that
\[|f(y)|\leq \|f\|_2 \|\phi_y\|_2\]
Furthermore, by the reproducing property 
\[\|\phi_y\|_2=\left(\int_S \phi_y(x)^2d\mu(y)\right)^{\frac{1}{2}} = \phi_y(y)^{1/2} = \sqrt{\frac{\dim H_j}{\mu(S)}}\]
proving the inequality. Since $\phi_y(y)=\frac{\dim(H_j)}{\mu(S)}$ we see that the equality is achieved when $f(x)=\phi_y(x)$ as claimed.
\end{proof}

\section{Linear harmonic hierarchies}

In this section we prove our main theoretical results, namely Theorem~\ref{Thm: polyhedra} which guarantees the existence of the harmonic hierarchies defined in Construction~\ref{Const: Spherequad} and Theorem~\ref{Thm: bounds} which gives quantitative bounds on their speed of convergence. Our first Lemma explains the key connection between representation theory and convolutions.

\begin{lemma}\label{lem: equiv} For an integer $s\geq 0$ define the linear map $\Gamma_{2s}$ sending a polynomial $f\in R$ to
\[\Gamma_{2s}(f)(x)=\int_S \langle x,y\rangle^{2s}f(y)d\mu(y).\]
The following statements hold:
\begin{enumerate}
\item The map $\Gamma_{2s}$ sends $R_{2k}$ into $R_{2s}$.
\item The map $\Gamma_{2s}$ is $SO(n)$-equivariant and in particular sends the subspace $\|x\|^{2(k-j)}H_{2j}\subseteq R_{2k}$ into the subspace $\|x\|^{2(s-j)}H_{2j}\subseteq R_{2s}$.
\item The map $\hat{\Gamma}_{2s,2k}(f):=\Gamma_{2s}(f)/\|x\|^{2(s-k)}$ is a well-defined linear endomorphism of $R_{2k}$.
\end{enumerate}
\end{lemma}
\begin{proof} $(1)$ By the multinomial theorem for every polynomial $f\in R$ we have
\[\Gamma_{2s}(f) = \sum_{(a_1,\dots, a_n): \sum a_i=2s}\binom{2s}{a_1,\dots, a_n} x_1^{a_1}\dots x_n^{a_n}\int_Sy_1^{a_1}\dots y_n^{a_n}f(y)d\mu(y)\]
which is an element of $R_{2s}$. $(2)$ 
For any $g\in O(n)$ and any $f\in R_{2k}$ we have
\[\rho^*(g)\Gamma_{2s}(f)= \Gamma_{2s}(f)(g^{-1}x)=\int_{S} \langle g^{-1}(x),y\rangle^{2s}f(y)d\mu(y)=\]
making the change of variables $y=g^{-1}(z)$ 
we obtain
\[=\int_{S}\langle g^{-1}(x),g^{-1}(z)\rangle^{2s}f(g^{-1}(z))d\mu(z)= \int_{S}\langle x,z\rangle^{2s}f(g^{-1}(z))d\mu(z)\]
where the second equality follows from the orthogonality of the matrix $g$. Since the last term equals ${\Gamma}_{2s}(\rho^*(g)(f))$ we conclude that $\Gamma_{2s}$ is a morphism of representations and therefore it must map the corresponding isotypical components to each other finishing the proof of $(2)$. Claim $(3)$ is immediate if $s<k$ since the map results from composing with multiplication by a fixed polynomial. If $s\geq k$ then by $(2)$ the subspace $\Gamma_{2s}(R_{2k})$ is contained in the multiples of $\|x\|^{2(s-k)}$ inside $R_{2s}$ proving that the ratio is well-defined.
\end{proof}

\begin{proof}[Proof of Theorem~\ref{Thm: polyhedra}] Since the evaluation at any point ${\rm ev}_x: R_{2k}\rightarrow \RR$ is a linear map the set $Q$, defined by the nonnegativity of finitely many evaluation functions is a polyhedral cone in $R_{2k}$. Using the notation of Lemma~\ref{lem: equiv} part $(3)$ the map $\hat{\Gamma}_h$ can be written as $\hat{\Gamma}_h=\sum_{j=0}^s a_{2j} \hat{\Gamma}_{2j,2k}$ and is therefore well-defined and linear. As a result the set $A:=\hat{\Gamma}_h(Q)$ is also a polyhedral cone in $R_{2k}$. 

Now suppose $f\in Q$, meaning that $f\in R_{2k}$ is nonnegative at all points $X$ of a cubature rule $(X,W)$ of algebraic degree $2(s+k)$ for $\mu$ and we wish to prove that $\hat{\Gamma}(f)$ is a nonnegative polynomial. If $x$ is any point in $S$ then
\[\hat{\Gamma}_h(f)(x)=\int_S\sum_{j=0}^s a_{2j}\langle x,t\rangle^{2j} f(y)d\mu(y)= \int_S\sum_{j=0}^s a_{2j}\langle x,t\rangle^{2j}\|y\|^{2(s-j)} f(y)d\mu(y)\]
where the last equality holds since $y$ is integrated over $S$ where $\|y\|=1$. As a function of $y$ the rightmost integrand is a homogeneous polynomial of degree $2(s+k)$ and we can therefore compute the integral using our cubature rule
\[\int_S\sum_{j=0}^s a_{2j}\langle x,t\rangle^{2j}\|y\|^{2(s-j)} f(y)d\mu(y) = \sum_{z\in X} W(z)h(\langle x,y\rangle) f(z).\]
The rightmost quantity is nonnegative since it is a sum of nonnegative terms because $g$ is nonnegative in the range $[-1,1]$ of $\langle x,y\rangle$, $f\in Q$ and the cubature weights are positive.
\end{proof}

The $SO(n)$-equivariance of the maps $\hat{\Gamma}_h$ (property $(2)$ of Lemma~\ref{lem: equiv}) and the fact that the decomposition of $R_{2k}$ into irreducibles is multiplicity-free already implies that averaging operators must diagonalize in the harmonic basis. We now prove Lemma~\ref{lem: diagonal} which gives an explicit diagonalization.

\begin{proof}[Proof of Lemma~\ref{lem: diagonal}] If $f\in R_{2k}$ is of the form $f=\|x\|^{2(k-\ell)} f_{2\ell}$ for some $f_{2\ell}\in H_{2\ell}$ and $x\in S$ is a point with $f(x)\neq 0$, then we have
\[\hat{\Gamma}_h(f)(x)=\sum_{j=0}^{2s}\lambda_{2j} \int_S g_{2j}(\langle x,y\rangle)f_{2\ell}(y)d\mu(y).\]
By the explicit formula in Theorem~\ref{Thm: Zonal} and definition~\ref{def: normGegenbauer} of normalized Gegenbauer polynomial we know that the equality \[g_{2j}(\langle x,y\rangle)=\phi^{(2j)}_x(y)\]
holds for all $x,y\in S$ and every index $j$. As a result, the reproducing propery of $\phi_x(y)$ and the mutual orthogonality of $H_j$ and $H_t$ for $t\neq j$ imply that 
\[\sum_{j=0}^{2s}\lambda_{2j} \int_S g_{2j}(\langle x,y\rangle)f_{2\ell}(y)d\mu(y) = \lambda_{2\ell} f_{2\ell}(x) \]
from which we know that $\hat{\Gamma}_h(f)=\lambda_{2\ell} f$ since $f(x)\neq 0$. 
\end{proof}

Now let $h(t)$ be an even univariate polynomial which is nonnegative on $[-1,1]$ of degree $s$ and assume $h(t)=\sum_{j=0}^{2s}\lambda_{2j}g_{2j}(t)$ be its unique representation in terms of Gegenbauer polynomials. Recall that the Frobenius threshold of $h(t)$ in degree $2k$ is given by
\[\tau_{2k}(h):=\sqrt{\sum_{j=0}^{2k} \dim(H_{2j})\left(\frac{1}{\lambda_{2j}}-1\right)^2}.\]

The following Lemma shows that the Frobenius threshold of a polynomial controls the distance between its inverse averaging operator and the identity.

\begin{lemma}\label{lem: FrobTineq} Assume $\hat{\Gamma}_h: R_{2k}\rightarrow R_{2k}$ is invertible. For every $f\in R_{2k}$ the following inequalities hold
\[\|\hat{\Gamma}_h^{-1}(f)-f\|_{\infty}\leq \frac{\tau_{2k}(h)}{\sqrt{\mu(S)}}\|f\|_2\]
and
\[  
\|\hat{\Gamma}_h^{-1}(f)-f\|_{\infty}\geq \frac{\tau_{2k}(h)}{\sqrt{\mu(S)}}\min_{j}\frac{\|f_{2j}\|_2}{\sqrt{\dim(H_{2j})}}.\]
\end{lemma}
\begin{proof} If $f\in R_{2k}$ has the harmonic expansion
\[f= \|x\|^{2k}f_0 + \|x\|^{2(k-1)}f_2 + \|x\|^{2(k-2)}f_4 +\dots+ f_{2k},\] 
then Lemma~\ref{lem: diagonal} implies that for any $z\in S$ the equality 
\begin{equation}
\label{eq: initial}
\Gamma_h^{-1}(f)(z)-f(z) = \sum_{j=0}^k \left(\frac{1}{\lambda_{2j}}-1\right) f_{2j}(z)  
\end{equation}
holds. By the Cauchy-Schwartz inequality this quantity is bounded above by
\[\sqrt{\sum_{j=0}^k \left(\frac{1}{\lambda_{2j}}-1\right)^2{\rm dim}(H_{2j})}\sqrt{\sum_{j=0}^k\frac{f_{2j}^2(z)}{{\rm dim}(H_{2j})}} \]
Since the $f_{2j}$ are harmonic, Lemma~\ref{lem: sup_bound} implies that the inequality 
\[\|f_{2j}\|_{\infty}^2\leq \frac{\dim(H_{2j})}{\mu(S)}\|f_{2j}\|_2^2\] 
holds and therefore~(\ref{eq: initial}) is bounded above by
\[\sqrt{\sum_{j=0}^k \left(\frac{1}{\lambda_{2j}}-1\right)^2{\rm dim}(H_{2j})}\sqrt{\sum_{j=0}^k\frac{\|f_{2j}^2\|_2^2}{\mu(S)}} = \tau_{2k}(h) \frac{\|f\|_2}{\sqrt{\mu(S)}}\]
as claimed. For the lower bound note that by~(\ref{eq: initial}) the following inequality holds for every $z\in S$

\[|\Gamma_h^{-1}(f)-f(z)|^2 \geq \left(\sum_{j=0}^k \sqrt{{\rm dim}(H_{2j})}\left(\frac{1}{\lambda_{2j}}-1\right) \frac{f_{2j}(z)}{\sqrt{{\rm dim}(H_{2j})}}\right)^2\]
integrating both sides over the sphere and dividing by $\mu(S)$ we conclude that 

\[\|\Gamma_h^{-1}(f)-f(z)\|_{\infty}^2 \geq \frac{1}{\mu(S)}\sum_{j=0}^k {\rm dim}(H_{2j})\left(\frac{1}{\lambda_{2j}}-1\right)^2 \frac{\|f_{2j}\|_2^2}{{\rm dim}(H_{2j})}\]

where we have used the fact that the $f_{2j}$ are pairwise orthogonal. We conclude that

\[\|\Gamma_h^{-1}(f)-f(z)\|_{\infty}^2 \geq \frac{\tau_{2k}^2(h)}{\mu(S)}\min_j \frac{\|f_{2j}\|_2^2}{{\rm dim}(H_{2j})}\]
which taking square roots is equivalent to the claimed lower bound.

\end{proof}

\begin{proof}[Proof of Theorem~\ref{Thm: bounds}] $(1)$ Follows immediately from Theorem~\ref{Thm: polyhedra} applied to the given sequence of polynomials $(h_s)_{s\in \mathbb{N}}$. Assume $\hat{\Gamma}_{h_s}$ is invertible and let $f\in R_{2k}$. For any $z\in X$ we have
\[\hat{\Gamma}_{h_s}^{-1}(f)(z)= \hat{\Gamma}_{h_s}^{-1}(f)(z)-f(z)+f(z)\geq  \min_{z\in X}f(z) - \|\hat{\Gamma}_{h_s}^{-1}(f)-f\|_{\infty}.\] 
If $f$ satisfies the hypothesis of $(2)$ then the rightmost term is strictly positive and therefore $\hat{\Gamma}_{h_s}^{-1}(f)\in Q_s$ because it is nonnegative at all nodes of the quadrature rule and therefore $f\in A_s$ as claimed. $(3)$ If $f$ is a strictly positive polynomial on $S$ then by compactness of the sphere it achieves a strictly positive minimum $\alpha^*$. By part $(2)$ the polynomial $f$ belongs to $A_s$ whenever $\alpha^*>\frac{\tau_{2k}(h_s)}{\mu(S)}\|f\|_2$ which happens for all sufficiently large $s$ since $\tau_{2k}(h_s)\rightarrow 0$ as $s\rightarrow \infty$. \end{proof}

\subsection{Optimization-free lower bounds for polynomial minimization.}\label{Sec: optimization-free}

Suppose $f\in R_{2k}$ and let $\alpha^*:=\min_{x\in S}f(x)$. Assume $h(t)$ is a univariate, even, nonnegative polynomial of degree $2s$ with $h(0)=1$ and such that $\hat{\Gamma}_h:R_{2k}\rightarrow R_{2k}$ is invertible. Assume $(X,W)$ is a cubature rule of algebraic degree $2(k+s)$. 
As an application of the theory developed so far we will obtain optimization-free lower bounds $\beta\leq \alpha^*$ via the following steps:

\begin{enumerate}
\item Compute a harmonic decomposition for $f$
\[ f= \|x\|^{2k}f_0 + \|x\|^{2(k-1)}f_2 + \|x\|^{2(k-2)}f_4 +\dots+ f_{2k}\] 
\item Compute the coefficients $\lambda_{2j}$ of the expansion of $h$ in terms of normalized Gegenbauer polynomials. By our assumptions $\lambda_0=1$ and that $\lambda_{2j}\neq 0$ for $j=0,\dots, k$.
\item Compute the polynomial $F:=\hat{\Gamma}_{g}^{-1}(f)$ with the formula
\[ F= \|x\|^{2k}f_0 + \frac{1}{\lambda_2}\|x\|^{2(k-1)}f_2 + \frac{1}{\lambda_4}\|x\|^{2(k-2)}f_4 +\dots+ \frac{1}{\lambda_{2k}}f_{2k}\] 
\item Evaluate $F(z)$ for $z\in X$ and let $\beta^*:=\min_{z\in X}F(z)$ be the smallest of those values.
\end{enumerate}

\begin{lemma} The inequality $\beta^*\leq \alpha^*$ holds.
\end{lemma}
\begin{proof} By construction the polynomial $p:=F-\beta^*\|x\|^{2k}$ is nonnegative at all cubature nodes $X$ and our cubature rule has algebraic degree $2(k+s)$. By Theorem~\ref{Thm: polyhedra} we conclude that $\hat{\Gamma}_h(p)\in A$ and is in particular a nonnegative polynomial. Since $F=\hat{\Gamma}_h^{-1}(f)$ and $\hat{\Gamma}_h(\|x\|^{2k})=\|x\|^{2k}$ because $\lambda_0=1$ we conclude that $\hat{\Gamma}_h(p)=f-\beta^*\|x\|^{2k}$ proving that $f$ is bounded below by $\beta^*$.
\end{proof}

\begin{remark} The number $\beta^*$ coincides with the optimum value of the linear optimization problem $\sup\left\{\lambda: f(x)-\lambda \in A\right\}$ because the point evaluations at the cubature nodes contain the extreme rays of the polyhedron $Q$. Since enumerating the cubature nodes is necessary to formulate the underlying linear optimization problem our optimization-free algorithm is equally accurate and computationally less expensive than linear optimization. 
\end{remark}

\subsection{Kernel selection}\label{Sec: KernelSelection}

In this section we address the problem of choosing the polynomial sequence $(h_s(t))_{s\in \mathbb{N}}$ so the resulting Harmonic Hierarchy converges quickly. We begin by proving Corollary~\ref{cor: various_g} which illustrates that the chosen sequence has indeed  a drastic effect on the rate of convergence.

\begin{proof}[Proof of Corollary~\ref{cor: various_g} (1)] In \cite{BConv} Blekherman explicitly calculates the coefficients of $h_s(t)=\tfrac{t^{2s}}{\int_S y^{2s} d\mu(y)}$ as a linear combination of Gegenbauer polynomials: 
\[
\lambda_{2j}^{(2s)} = \frac{s!\Gamma(\frac{2s+n}{2})}{(s-j)!\Gamma(\frac{2s+2j+n}{2})},
\]
here $\Gamma$ denotes the usual gamma function. When $j=0$, $\lambda_{0}^{(2s)} = 1$.  For $j>0$, the recursion property of the gamma function gives
\[
	\lambda_{2j}^{(2s)} = \frac{s!\prod\limits_{t=1}^s(t+\frac{n}{2})\Gamma(\frac{n}{2})}{(s-j)!\prod\limits_{t=1}^{s+j}(t+\frac{n}{2})\Gamma(\frac{n}{2})} = \frac{s(s-1)\dots (s-j+1)}{(s+j+\frac{n}{2})(s+j-1+\frac{n}{2})\dots(s+1+\frac{n}{2})}.
\]
By factoring $s$ on all terms and separating the product suitably, we can rewrite
\[
	\lambda_{2j}^{(2s)} = \frac{1}{1+\frac{j+n/2}{s}} \frac{1-\frac{1}{s}}{1+\frac{j-1+n/2}{s}} \dots \frac{1-\frac{j-1}{s}}{1+\frac{1+n/2}{s}}.
\]
Let us now consider the logarithm of $1/\lambda_{2j}^{(2s)}$
\[
\log\left(\frac{1}{\lambda_{2j}^{(2s)}}\right) = \sum\limits_{t=1}^j \log\left( 1+\frac{t+\frac{n}{2}}{s} \right) - \sum\limits_{t=1}^{j-1} \log\left(1-\frac{t}{s}\right),
\]
a Taylor expansion of the previous terms yields the following
\begin{align*}
	\log\left(\frac{1}{\lambda_{2j}^{(2s)}}\right) = \frac{1}{s}\left(j^2 +  \frac{jn}{2} \right) +  O\left( \frac{1}{s^2} \right).
\end{align*}
With the previous approximation and yet another Taylor expansion we obtain
\[
\frac{1}{\lambda_{2j}^{(2s)}} -1 \approx  e^{ \frac{j^2+\frac{jn}{2} }{s}  } -1 =  \frac{j^2+\frac{jn}{2} }{s}  + O\left( \frac{1}{s^2} \right),
\]
except when $j=0$, in which case ${1}/{\lambda_{0}^{(2s)}} -1=0$. Now, using the previous analysis on the Frobenius threshold results in 
\[
\tau_{2k}(h_s)=\sqrt{\sum_{j=0}^{k} \dim(H_{2j}) \left(\frac{1}{\lambda_{2j}^{(2s)}}-1\right)^2} \approx \sqrt{\sum_{j=1}^{k} \left(\frac{j^2+\frac{jn}{2} }{s}  + O\left( \frac{1}{s^2} \right) \right)^2}
\]
which is bounded above and below in the following way
\[
\frac{1+\frac{n}{2}}{s}+O\left( \frac{1}{s^2} \right) \leq \tau_{2k}(h_s) \leq  D_{2k}\frac{k^2+\frac{kn}{2} }{s} + O\left( \frac{1}{s^2} \right),
\]
where $D_{2k}=\max_{j=0,\dots, k} \dim(H_{2j})$.
\end{proof}

The proof of Corollary~\ref{cor: various_g} (2) requires a bit more work. Suppose  \[
h_s(t)=q_s(t)^2=\sum\limits_{j=0}^{2s} \lambda_j g_j(t),
\]
where $q_s(t) = \sum_{j=0}^s \eta_j g_j(t).$ Since Gegenbauer polynomials form an orthogonal base (with respecto to the weighted 2-norm), we have that
\[
\lambda_\ell = \frac{\int\limits_{-1}^1 q_s(t)^2 g_\ell(t) w(t) dt }{\frac{\dim(H_\ell)^2N_{\ell}^2}{C_\ell(1)^2\mu(S)^2}}
\]
since $\int_{-1}^1 g_\ell(t)^2 w(t)dt = \tfrac{\dim(H_\ell)^2N_{\ell}^2}{C_\ell(1)^2\mu(S)^2}$ where $N_{\ell}=\int_{-1}^1 C_{\ell}(t) w(t)dt$. Define the Toeplitz matrix $\tau[f]$ of a polynomial $f$ as the $s\times s$ matrix with $(i,j)$-th coordinate
\[
\tau[f]_{ij} := \int\limits_{-1}^1 g_i(t) g_j(t) f(t) w(t) dt,
\]
and define $A_\ell = \tau[g_\ell]$.
Note then that
\begin{align}
	\lambda_\ell = \frac{C_\ell(1)^2\mu(S)^2}{\dim(H_\ell)^2N_{\ell}^2} \eta^t A_\ell \eta. \label{eq:eigenvalues}
\end{align}
Now, we are interested in minimizing the Frobenius threshold over all $q_s$
\[
\tau^*_{2k,s}=\min_{q_s} \tau_{2k}(h_s)  = \min_{q_s} \sqrt{ \sum\limits_{j=0 }^k \dim(H_{2j}) \left( \frac{1}{\lambda_{2j}} -1 \right)^2 } 
\]
under the contraint $\lambda_0=1$. We will now attempt to find an upper bound on $\tau_{2k,s}^*$ with the solution to the alternative problem 
\[
\rho^*_{2k,s} = \min_{q_s,\lambda_0=1} \sum\limits_{j=0}^k (1-\lambda_{2j}),
\]
which we will prove can be reformulated as an eigenvalue problem. First, by equation~(\ref{eq:eigenvalues})
\[
\sum\limits_{j=0}^k (1-\lambda_{2j}) =  \sum\limits_{j=0}^k \left(1- {\frac{C_{2j}(1)^2\mu(S)^2}{\dim(H_{2j})^2N_{2\ell}^2}} \eta^t A_{2j} \eta\right) = k -\eta^t \left( \sum\limits_{j=0}^k \frac{C_{2j}(1)^2\mu(S)^2}{\dim(H_{2j})^2N_{2\ell}^2} A_{2j}  \right) \eta ,
\]
thus 
\[
\rho^*_{2k,s} = \min_{q_s,\lambda_0=1} \sum\limits_{j=0}^k (1-\lambda_{2j})= k -  \max_{\eta,\lambda_0=1} \eta^t \left( \sum\limits_{j=0}^k \frac{C_{2j}(1)^2\mu(S)^2}{\dim(H_{2j})^2N_{2\ell}^2} A_{2j}  \right) \eta.
\]
This is strictly not an eigenvalue problem since $\eta$ is not necessarily restricted to normalized vectors, however we can properly rescale $\eta$ by noticing that
\[
(A_0)_{i,j} = \int \limits_{-1}^1 g_i(t) g_j(t) w(t) dt = \begin{cases}
	\frac{\dim(H_{j})^2N_{j}^2}{C_{j}(1)^2\mu(S)^2} & i=j,\\
	0 &  i\neq j,
\end{cases}
\]
or put in another way $A_0$ is diagonal with coefficients $\frac{\dim(H_{j})^2}{C_{j}(1)^2\mu(S)^2}$ for $0\leq j\leq s$. This means that
\[
\lambda_0 =\frac{C_0(1)^2\mu(S)^2}{\dim(H_0)^2N_{0}^2} \eta^t A_0 \eta= \sum\limits_{j=0}^s \frac{\dim(H_{j})^2N_{j}^2}{C_{j}(1)^2N_0^2} \eta_j^2 =1,
\]
so the change of basis $e_j=\frac{\dim(H_j)N_{j}}{C_j(1)N_0} \eta_j$, which is in fact $e= \frac{\mu(S)}{N_0} \sqrt{  A_0 } \eta$, lets us rewrite the constraint $\lambda_0=1$ as $||e||_2 = 1$. Consequently, the alternate problem becomes an eigenvalue problem
\begin{align*}
	\rho^*_{2k,s} &= k -  \max_{e,||e||_2 = 1} e^t \left(\frac{\mu(S)}{N_0} \sqrt{ A_0 }\right)^{-1} \left( \sum\limits_{j=0}^k \frac{C_{2j}(1)^2\mu(S)^2}{\dim(H_{2j})^2N_{2j}^2} A_{2j}  \right) \left(\frac{\mu(S)}{N_0} \sqrt{ A_0 }\right)^{-1} e\\
	&= k - k \max_{e,||e||_2 = 1} e^t  \left( \frac{1}{k}\sum\limits_{j=0}^k \frac{C_{2j}(1)^2N_0^2}{\dim(H_{2j})^2N_{2j}^2} A_0^{-1} A_{2j}  \right)  e\\
	&= k - k \lambda_{max}(T_{2k,s})
\end{align*}
where 
\[
T_{2k,s}= \frac{1}{k}\sum\limits_{j=0}^k \frac{C_{2j}(1)^2N_0^2}{\dim(H_{2j})^2N_{2j}^2} A_0^{-1} A_{2j}
\] 
and $\lambda_{max}(T_{2k,s})$ is its maximum eigenvalue. Hence, the optimal $\rho_{2k,s}^*$ is realized by any normalized eigenvector $e^*$ of $T_{2k,s}$ for the eigenvalue $\lambda_{max}(T_{2k,s})$. Furthermore, Fang and Fawzi proved that the choice of $\lambda_{\ell}$ corresponding to $e^*$ converges to 0 as $s\to\infty$ with a rate of $1/s^2$. More precisely:

\begin{theorem}(Fang-Fawzi~\cite{FangFawzi}, Proposition 7)
	The matrix $T_{2k,s}$ satisfies $$\lambda_{max}(T_{2k,s})\geq 1-kn^2O(\tfrac{1}{s^2}),$$ thus $\rho^*_{2k,s} \leq k^2n^2 O(\tfrac{1}{s^2})$.
\end{theorem}
To see how this result connects to our original problem, we exhibit the relation between $\tau_{2k,s}(h_s)$ and $\rho_{2k,s}^*=\rho_{2k,s}(h_s)$ when $h_s$ is the optimal kernel obtained by minimizing $\rho_{2k,s}$.
\begin{lemma}
	If $\rho^*_{2k,s} < 1$  then $\tau_{2k}(h_s) \leq \sqrt{D_{2k}}\frac{\rho_{2k,s}^*}{1-\rho_{2k,s}^*}$, in particular $\tau_{2k,s}^* \leq \sqrt{{D_{2k}}}\frac{\rho_{2k,s}^*}{1-\rho_{2k,s}^*}$.
\end{lemma}
\begin{proof}
First of all, the following inequality holds
\[
\tau_{2k}(h_s) \leq \sqrt{D_{2k}} \sqrt{\sum\limits_{j=0}^k \left( \frac{1}{\lambda_{2j}} - 1 \right)^2} \leq \sqrt{D_{2k}} \sum\limits_{j=0}^k \left| \frac{1}{\lambda_{2s}} -1 \right|.
\]	
Note that for every $\ell=0,1,\dots, k$, $1-\lambda_{2\ell} \leq \rho_{2k,s}^* = \sum_{j=0}^k (1-\lambda_{2j})$, then $\lambda_{2\ell} \geq 1-\rho_{2k,s}^*>0$. It follows that
\[
\sum\limits_{j=0}^k \left| \frac{1}{\lambda_{2s}} -1 \right| =  \sum\limits_{j=0}^k \frac{1-\lambda_{2s}}{\lambda_{2s}} \leq \frac{1}{1-\rho_{2k,s}^*} \sum\limits_{j=0}^k (1-\lambda_{2j}) = \frac{\rho_{2k,s}^*}{1-\rho_{2k,s}^*}.
\]
The desired inequality is obtained by putting the previous two together.
\end{proof}

We can now complete the proof of Corollary~\ref{cor: various_g} (2).

\begin{proof}[Proof of Corollary~\ref{cor: various_g} (2)] 
Since $\rho^*_{2k,s} \leq k^2n^2 O(\tfrac{1}{s^2})$, there exists an $S$ such that $\rho^*_{2k,s}\leq 1/2$ for $s\geq S$. Consequently, if $s\geq S$ the following holds
\[
\tau_{2k}(h_s) \leq \sqrt{D_{2k}} \frac{\rho_{2k,s}^*}{1-\rho_{2k,s}^*} \leq 2\sqrt{D_{2k}} \rho_{2k,s}^*\leq 2\sqrt{D_{2k}} k^2n^2 O(\tfrac{1}{s^2}). 
\]
In particular, as $s\to\infty$, $\tau_{2k}(h_s) \leq  \sqrt{D_{2k}} k^2n^2 O(\tfrac{1}{s^2})$.
\end{proof}

To summarize, we can obtain a polynomial sequence $(h_s)_s$ such that their corresponding Frobenius thresholds converge to 0 with a rate of $1/s^2$ as $s\to \infty$. Furthermore, this sequence can be calculated explicitly (and free of optimization) in terms of the solution to an eigenvalue problem. Specifically, for a fixed $s$ the coefficients $\lambda_{\ell}$ of $h_s$ in the basis of normalized Gegenbauer polynomials are given by 
\begin{align}
	\lambda_\ell = \frac{C_\ell(1)^2 N_0^2}{\dim(H_\ell)^2N_\ell^2} (e^*)^t A_0^{-1}A_\ell e^*, \label{eq:fawzi-eigen}
\end{align}
where $e^*$ is any normalized eigenvector of $T_{2k,s}$ for its maximum eigenvalue $\lambda_{max}(T_{2k,s})$.

We now assume a fixed degree $2s$ (expressing a limited amount of available computational resources) and ask whether we can choose an {\it optimal} polynomial $g(t)$ of degree $2s$, in the sense of having minimal Frobenius threshold $\tau_{2k}(g)$. Our main result is Theorem~\ref{Thm: Optimal_g} which shows that this problem is essentially a convex optimization problem (and thus amenable to standard techniques~\cite{Bertsekas}).

\begin{lemma}\label{lem: moduli_g} The set $\calG_{2s}$, consisting of univariate polynomials $g(t)$ of degree $\leq 2s$ that are  even, nonnegative in $[-1,1]$ and satisfy
$g(0)=\frac{1}{\mu(S)}$ is a semidefinitely representable set.
\end{lemma}
\begin{proof} A polynomial $h(t)$ is nonnegative in $[-1,1]$ iff it can be written as $h(t)=s_1(t)+(x+1)b_1(t)+(1-x)b_2(t)$ where $s_1$, $b_1$ and $b_2$ are sums of squares of polynomials of degrees at most $s$,$s-1$ and $s-1$ respectively. Equivalently such polynomials are the image of the triples $(A,B_1,B_2)$ of symmetric positive semidefinite matrices with $s+1$, $s$ and $s$ rows respectively, under the linear map
\[\pi(A,B_1,B_2) = \vec{m}^tA\vec{m}+(x+1)\vec{n}^tB_1\vec{n}+ (1-x)\vec{n}^tB_2\vec{n}\]
where $\vec{m}$ (resp. $\vec{n}$) is the vector of monomials $(1,t,t^2,\dots,t^s)$ (resp. $(1,t,t^2,\dots,t^{s-1})$).
As a result the set $N$ of polynomials $h(t)$ nonnegative in $[-1,1]$ of degree $\leq 2s$ is semidefinitely representable. If $h(t)$ is any such polynomial then $\frac{h(t)+h(-t)}{2}$ is its even part and since this operation is linear in $h$ we conclude that the set $W$ of even polynomials in $N$ is also an SDr set. The condition that any such polynomial has a prescribed value at $0$ is linear and therefore $\calG_{2s}$ is the intersection of $W$ with an affine hyperplane and therefore an SDr set as claimed.
\end{proof}

\begin{theorem}\label{Thm: Optimal_g} For all sufficiently large integers $s$ the optimization problem 
\[\min_{g\in \calG_{2s}} \tau_{2k}(g)\] is equivalent to a convex programming problem.
\end{theorem}
\begin{proof} The coefficients $\lambda_{2j}(g)$ expressing a polynomial $g$ as a linear combination of weighted Gegenbauer polynomials are a linear function of $g(t)$ because they can be computed by integration, using the well-known orthogonality of Gegenbauer polynomials.  Furthermore, the function $(\frac{1}{\lambda}-1)^2$ is a convex function of $\lambda$ for $0<\lambda<3/2$. We conclude that the Frobenius Threshold $\tau_{2k}(g)$ is a convex function of $g$ in the convex SDr set $\calG_{2s}'$ consisting of the polynomials $\calG_{2s}$ such that $0\leq \lambda_{2j}(g)\leq \frac{3}{2}$. If any of these inequalities  fails then the inequality $\tau_{2k}(g)\geq \frac{1}{3}$ holds so any optima of the problem must lie in the set $\calG_{2s}'$ for sufficiently large $s$ as claimed.
\end{proof}

\subsection{Harmonic hierarchies for moment problems.}\label{Sec: HHMoments}

As mentioned in the introduction, Tchakaloff's Theorem~\cite{Tchakaloff} proves that the cone $P_{2k}^*\subseteq R_{2k}^*$ dual to $P_{2k}$ consists of the moment operators of degree $2k$ for all Borel measures on $S$. More precisely, the elements of $P_{2k}^*$ are the linear operators $\ell: R_{2k}\rightarrow \RR$ which satisfy
\[\forall f\in R_{2k}\left(\ell(f)=\int_S f(y)d\nu(y)\right)\]
for some Borel measure $\nu$ on $S$. It is therefore a problem of much interest to characterize or to approximate $P_{2k}^*$. For every integer $s$, Construction~\ref{Const: HH} provides us with polyhedral cones $A_s$ and $Q_s$ in $R_{2k}$ satisfying the inclusions
\[A_s\subseteq P_{2k}\subseteq Q_s\]
Their dual convex cones in $R_{2k}^*$ therefore satisfy
\[A_s^*\supseteq P_{2k}^*\supseteq Q_s^*\]
providing us with an (outer) {\it Harmonic Hierarchy for moments} $(A_s^*)_{s\in \mathbb{N}}$. The following Theorem provides a description of the cones $A_s^*$ and $Q_s^*$ amenable to computation. For a point $y\in S$ define the operator $L_y\in R_{2k}^*$ as
\[L_{y}:=\left\langle \sum_{j=0}^k \lambda^{(s)}_{2j}\|x\|^{2(k-j)}\phi^{2j}_y(x),\bullet\right\rangle\]
where the $\lambda_{2j}^{(s)}$ are the coefficients of $g_s$ in its Gegenbauer expansion (as in Lemma~\ref{lem: diagonal}) and $\phi_y^{2j}(x)$ is the homogeneous polynomial defined in Theorem~\ref{Thm: Zonal}.

\begin{theorem}\label{Thm: HHMoments}
The following statements hold: 
\begin{enumerate}
\item 
For every positive integer $s$ and $d_s:={\rm deg}(g_s)$ we have:
\begin{enumerate}
\item The polyhedral cone $Q_s^*$ is the convex hull of the point evaluations at the cubature nodes $X_{k+d_s}$.

\item If $\hat{\Gamma}_{g_s}$ is invertible then the polyhedral cone $A_s^*$ is given by \[A_s^*={\rm Conv}\left(\left\{L_y: y\in X_{k+d_s}\right\}\right).\]
\end{enumerate}
\item If $\lim_{s\rightarrow \infty} \tau_{2k}(g_s)=0$ then the hiererachy $(A_s^*)_{s\in \mathbb{N}}$ and $(Q_s^*)_{s\in \mathbb{N}}$ converge to $P_{2k}^*$ in the sense that the following equalities hold
\[\bigcap_{s=0}^{\infty} A_s^* = P_{2k}^*=\overline{\bigcup_{s=0}^{\infty} Q_s^*}.\]
\end{enumerate}
\end{theorem}
\begin{proof} $(1a)$ The cone $Q_s$ is defined by nonnegativity of the evaluations at cubature nodes $X_{k+d_s}$. $(1b)$ By Theorem~\ref{Thm: Zonal} for every polynomial $f\in R_{2k}$ and $y\in S$ the equality $L_y(f)=\hat{\Gamma}_{g_{s}}(f)(y)$
holds. It follows that $f\in {\rm Conv}(L_y:y\in X_{k+d_s})^*$ if and only if $\hat{\Gamma}_{g_{s}}(f)\in Q_s$ proving the claim by the bi-duality Theorem of convex geometry. Part $(2)$ is immediate from Theorem~\ref{Thm: bounds} and bi-duality.
\end{proof}

\begin{proof}[Proof of Theorem~\ref{Harmonic_Moments}] It is immediate from Theorem~\ref{Thm: HHMoments}.
\end{proof}

\section{A Julia package for harmonic hierarchies.}\label{Sec:Julia}

In this Section we show some numerical examples computed with our Julia package for Harmonic Hierarchies available at~\href{https://github.com/SergioCS147/HarmonicPolya.git}{github}. The package has the following capabilities:

\begin{enumerate}
\item Computing the Gauss-product cubature rules from Section~\ref{Sec: Gauss-product} for $S\subseteq \RR^n$ and any degree $2s$.
\item Computing the harmonic decomposition of polynomials using the algorithm of Axler and Ramey \cite{AR}. 
\item Computing the upper bound for polynomial minimization problems on spheres from the work of Martinez et al~\cite{Piazzon-Vianello-Q} discussed in Remark~\ref{Rem:cubature}. 
\item Computing our optimization-free lower bound for minimization problems on spheres (see Section~\ref{Sec: optimization-free}) using the kernels appearing in Corollary~\ref{cor: various_g}. 
\end{enumerate}

Figures~\ref{fig:upperb} and ~\ref{fig:squares_polys} respectively show upper and lower bounds for the minima of the Motzkin and Robinson polynomials calculated using our package. The lower bounds implemented both of the sequences $(g_s)_{s\in\mathbb{N}}$ of Corollary~\ref{cor: various_g} parts $(1)$ and $(2)$, figures~\ref{fig:1a} and ~\ref{fig:1b} respectively. 

The Motzkin and Robinson polynomials are given by the formulas
\[m(x_1,x_2,x_3) = x_1^2x_2^4+ x_1^4x_2^2+x_3^6-3x_1^2x_2^2x_3^2,\]
\begin{align*}r(x_1,x_2,x_3,x_4) &= x_1^2(x_1-x_4)^2 + x_2^2(x_2-x_4)^2 + x_3^2(x_3-x_4)^2 \\
	&\quad + 2x_1x_2x_3(x_1+x_2+x_3-2x_4),\end{align*}
respectively. These are well-known nonnegative polynomials with zeroes. The figures show that the practical behavior of our optimization-free lower bound closely mirrors the predicted theoretical behavior. 

\begin{figure}[b]
	\centering
	\includegraphics[width=0.65\linewidth]{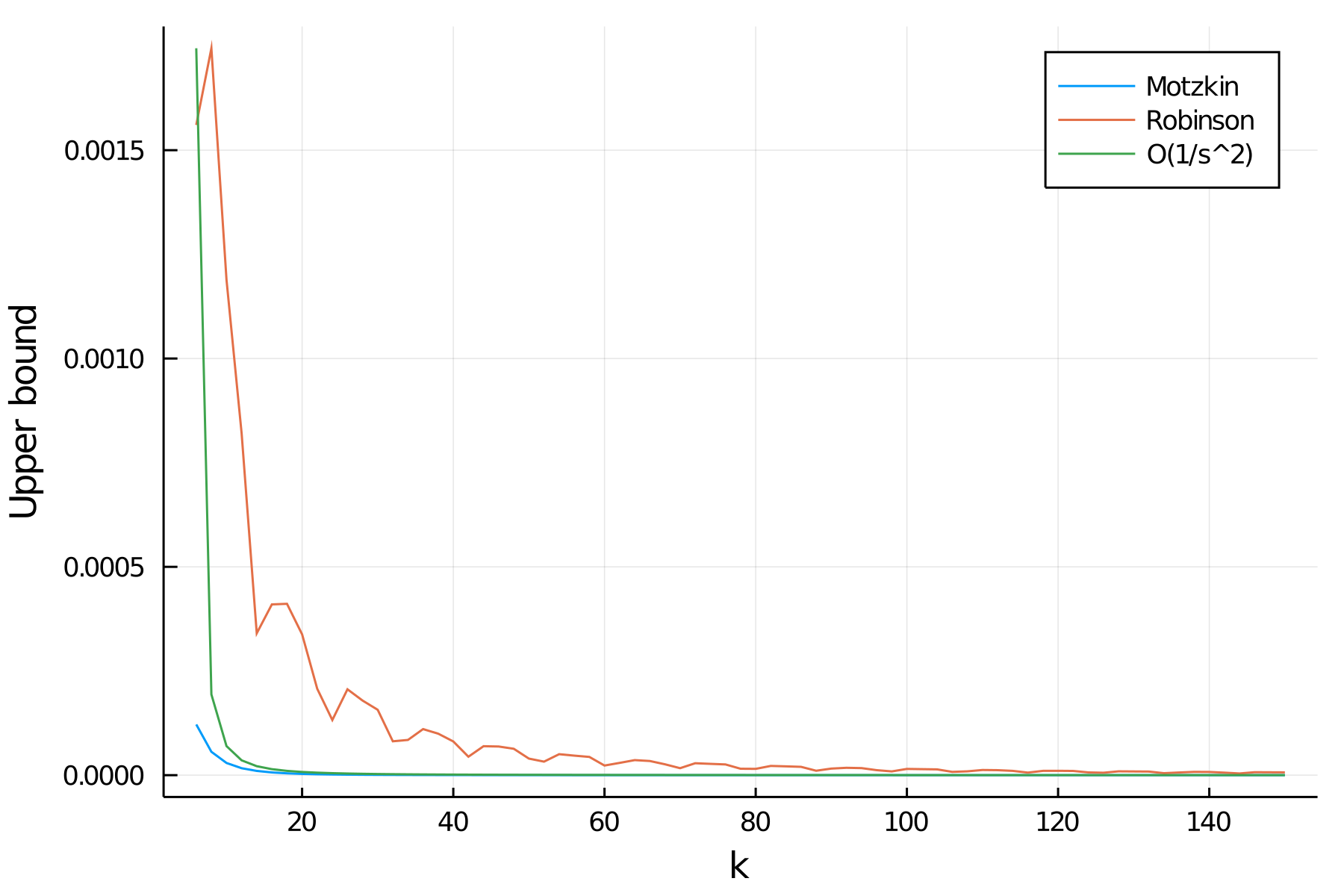}
	\caption{Figure 1 shows the upper bounds calculated for the Motzkin and Robinson polynomials employing our Julia package's implementation of the method described in Remark~\ref{Rem:cubature}. We also include the plot of a $O(-1/s^2)$ in order to compare its behavior to the theoretical convergence rate. }
	\label{fig:upperb}
\end{figure}

\begin{figure}[h!]
	\centering
	\begin{subfigure}{0.7\textwidth}
		\includegraphics[width=\linewidth]{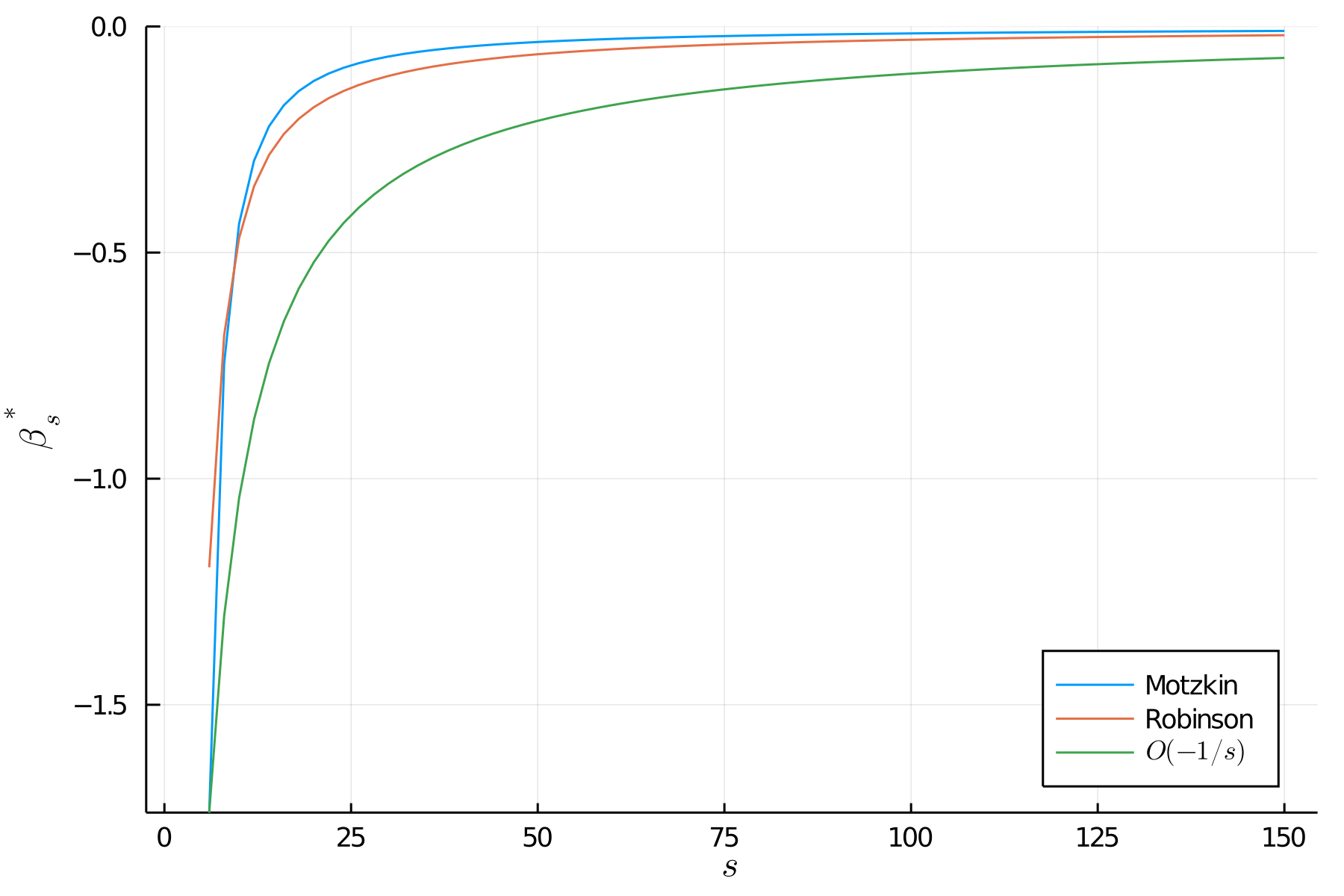}
		\caption{Using the pure powers sequence.} \label{fig:1a}
	\end{subfigure}
	\begin{subfigure}{0.7\textwidth}
		\includegraphics[width=\linewidth]{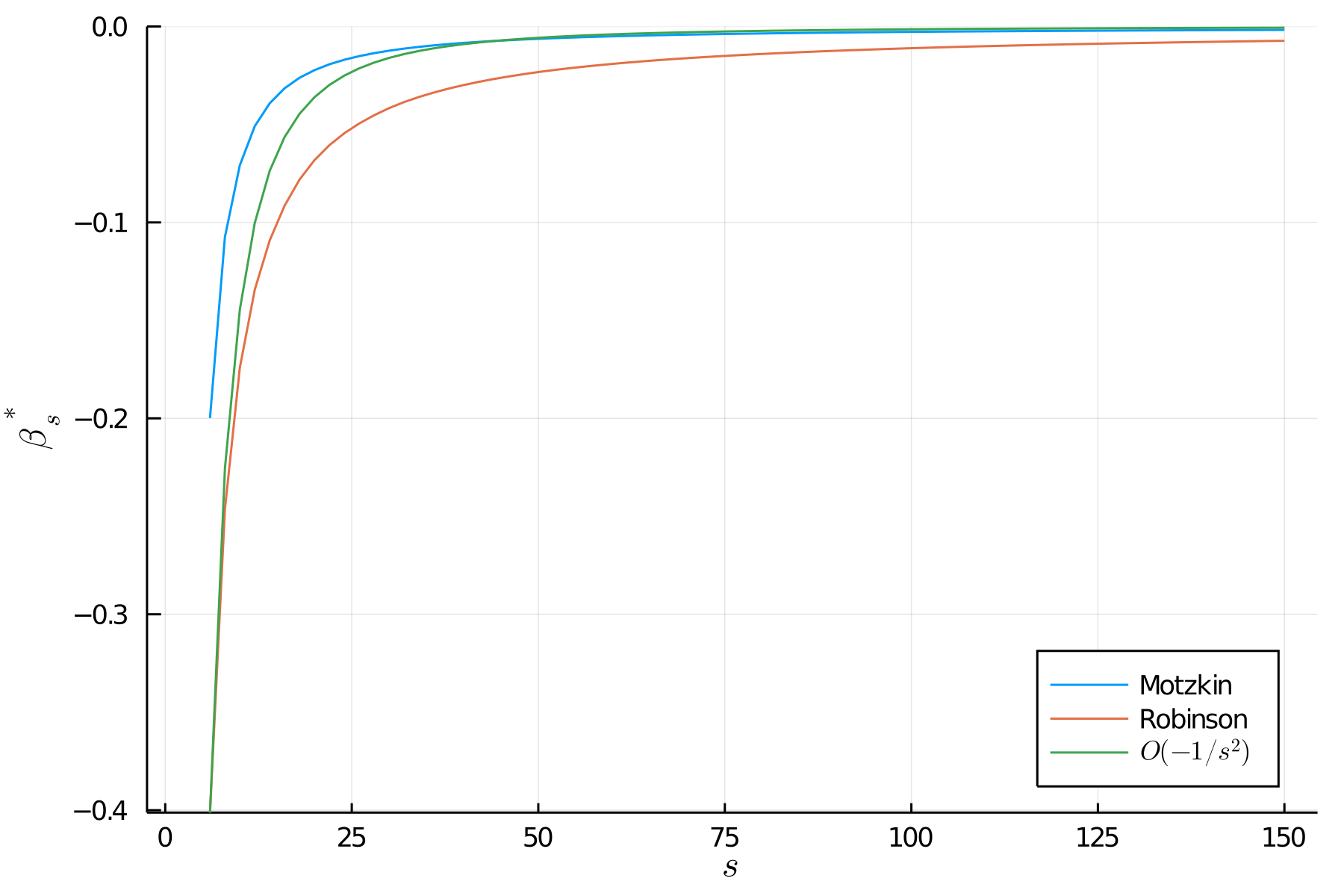}
		\caption{Using the Fang-Fawzi sequence.} \label{fig:1b}
	\end{subfigure}%
	\caption{Figures 2a and 2b show the lower bounds calculated for the Motzkin and Robinson polynomials employing our Julia package's implementation of the squares and Fang-Fawzi sequence of kernels from Corollary~\ref{cor: various_g}(1) and (2) respectively. We also include the plot of $O(-1/s)$ and $O(-1/s^2)$ functions respectively in order to compare the behavior of the obtained lower bounds and the theoretical convergence rate. }
	\label{fig:squares_polys}
\end{figure}

\end{document}